\documentclass[a4paper,12pt,oneside,dvipsnames,reqno]{amsart} 

\usepackage{amsmath,amsfonts,amsthm,amssymb}
\usepackage{xcolor}
\usepackage{graphicx}
\usepackage{tensor}							
\usepackage{cite}
\usepackage{mathtools}
\usepackage{mathrsfs}

\usepackage{hyperref}
\hypersetup{colorlinks=true, citecolor=PineGreen, linkcolor=RoyalBlue, linktoc=page} 

\usepackage{geometry}

\theoremstyle{plain}
\newtheorem{theorem}{Theorem}[section] 
\newtheorem{lem}[theorem]{Lemma}
\newtheorem{prop}[theorem]{Proposition}
\newtheorem{cor}[theorem]{Corollary}

\theoremstyle{remark}
\newtheorem{rem}[theorem]{Remark}

\theoremstyle{definition}
\newtheorem{defn}[theorem]{Definition}

\newcommand{\R}{\mathbb{R}}

\newcommand{\C}{\mathbb{C}}
\newcommand{\Cx}{\mathbb{C}^{\times}}

\newcommand{\Z}{\mathbb{Z}}

\newcommand{\Q}{\mathbb{Q}}

\newcommand{\Xf}{\mathfrak{X}}

\newcommand{\of}{\mathfrak{o}}

\newcommand{\ofx}{\mathfrak{o}^{\times}}

\newcommand{\pf}{\mathfrak{p}}

\newcommand{\As}{\mathscr{A}}

\newcommand{\SG}{\mathscr{SG}}

\newcommand{\Fx}{F^{\times}}

\newcommand{\Dx}{D^{\times}}

\newcommand{\GL}{\operatorname{GL}}

\newcommand{\Hom}{\operatorname{Hom}}

\newcommand{\isom}{\cong}

\newcommand{\tot}{\operatorname{tot}}

\newcommand{\Nrd}{\operatorname{Nrd}}

\newcommand{\Mod}[1]{\ (\operatorname{mod}\ #1)}

\newcommand{\sm}{\smallsetminus}

\newcommand{\abs}[1]{\lvert{#1}\rvert}

\newcommand{\arrup}[2]{\mathrel{\mathop{#1}^{#2}}}

\title{An Explicit Conductor Formula for $\GL_{n}\times\GL_{1}$}

\author{Andrew Corbett}
\address{Mathematisches Institut, Bunsenstr.\ 3-5, 37073 G{\"o}ttingen, Germany}
\email{andrew.corbett@uni-goettingen.de}
\date{30$^{\mathrm{th}}$ January 2019}


\begin{document}

\begin{abstract}
We prove an explicit formula for the conductor of an irreducible, admissible representation of $\GL_{n}(F)$ twisted by a character of $\Fx$ where the field $F$ is local and non-archimedean. As a consequence, we quantify the number of character twists of such a representation of fixed conductor.
\end{abstract}

\maketitle



\section{The problem of the twisted conductor}


Let $F$ denote a non-archimedean local field of characteristic zero and let $n\geq 2$. For an irreducible, admissible representation $\pi$ of $\GL_{n}(F)$ and a quasi-character $\chi$ of $\Fx$, we can form the twist $\chi\pi=(\chi\circ\det)\otimes\pi$. Our main result, Theorem \ref{thm:exact-formula}, is an explicit formula for the conductor $a(\chi\pi)$, equal to the Artin conductor, as defined in \S \ref{sec:preliminar-notions}. This formula is given by
\begin{equation}\label{eq:intro-formula}
a(\chi\pi)=a(\pi)+\Delta_{\chi}(\pi) - \delta_{\chi}(\pi)
\end{equation}
where $\Delta_{\chi}(\pi)$ and $\delta_{\chi}(\pi)$ are non-negative integers as defined in Theorem \ref{thm:exact-formula}; they denote a {dominant} and a {non-twist-minimal interference} term, respectively. We give detailed analysis of these terms in \S \ref{sec:leading-interference}, answering questions such as ``for what number of $\chi$ is there interference present?''

As an example, computing $a(\chi\pi)$ in the limit $a(\chi)\rightarrow\infty$ is straightforward: from Proposition \ref{prop:sq-int-formula} and Equation \eqref{eq:decomposition-a-chi-pi} we deduce that
\begin{equation}\label{eq:large-chi}
a(\chi\pi)=na(\chi)
\end{equation}
whenever $a(\chi)>a(\pi)$. In this case $\Delta_{\chi}(\pi)= na(\chi)-a(\pi)$ and $\delta_{\chi}(\pi)=0$. Bushnell--Henniart extend \eqref{eq:large-chi} by proving the upper bound\footnote{Inequality \eqref{eq:bushnell-bound} is a special case of both \cite[Theorem 1]{bushnell-henniart-upper-bound} and our main result, Theorem \ref{thm:exact-formula}. (See also Corollary \ref{cor:upper-and-lower-bounds} for a more precise inequality.)}
\begin{equation}\label{eq:bushnell-bound}
a(\chi\pi)\leq \max\{ a(\pi),\, a(\chi)\} + (n-1)a(\chi),
\end{equation}
surrendering to a weaker bound in the region $0\leq a(\chi)\leq a(\pi)$. Nevertheless, this bound is sharp in that it is attained for \textit{some} $\pi$ and $\chi$, as in \eqref{eq:large-chi} for example.


However, in general such examples become sparse, rendering \eqref{eq:bushnell-bound} as rather coarse as one averages over $\chi$ with $a(\chi)\asymp a(\pi)$. In such cases, evaluating the integers $\Delta_{\chi}(\pi)$ and $\delta_{\chi}(\pi)$ exactly is of crucial importance for numerous problems in analytic number theory. 

In this paper we consider applications to studying $a(\chi\pi)$ in a quantitative fashion. For example, we count the number of $\chi$ for which $a(\chi\pi)$ is equal to a given integer (see \S \ref{sec:analysis}). Such analysis would most commonly be applied when considering $a(\chi\pi)$ on average.

Our formula may be utilised when studying the analytic behaviour of automorphic $L$-functions. 
In particular, it is applicable in conjunction with the following two techniques: taking harmonic $\GL_{1}$-averages and applying the functional equation for $\GL_{n}\times\GL_{1}$-$L$-functions.
For example, conductors of such character twists arise in the work of Nelson--Pitale--Saha \cite{nps} who address the quantum unique ergodicity conjecture for holomorphic cusp forms with ``powerful'' level (see \cite[Remarks 1.9 \& 3.16]{nps}).
The current record for upper and lower bounds for the sup-norm of a Maa{\ss}-newform on $\GL_{2}$ in the level aspect \cite{saha-sup-norm,saha-hybrid,saha-large-values} also depends crucially on the $n=2$ case of Theorem \ref{thm:exact-formula}.

An instance where \eqref{eq:intro-formula} is applied constructively is carried out in \cite{corbett-saha}, once again when $n=2$. Originally, in \cite{brunault}, Brunault computed the value of ramification indices of modular parameterisation maps of various elliptic curves over $\Q$. Whenever the newform attached to $E$ is ``twist minimal'', Brunault could prove that this index was trivial (equal to $1$), holding in particular whenever the conductor of $E$ is square-free. This problem has now been completely solved by Saha and the present author \cite{corbett-saha}. In our solution, it is the degenerate cases of \eqref{eq:intro-formula}, with non-trivial $\Delta_{\chi}(\pi)$ and $\delta_{\chi}(\pi)$, that give rise to the few examples of non-trivial ramification indices.

These results all concern the case $n=2$, where the conductor formula for twists of supercuspidal representations was given by Tunnell \cite[Proposition 3.4]{tunnell} in his thesis (see \cite[Lemma 2.7]{corbett-saha} for the general case). Tunnell himself applied his formula to count isomorphism classes of supercuspidal representations of fixed odd conductor \cite[Theorem 3.9]{tunnell}. He used this observation in his proof of the local Langlands correspondence for $\GL_{2}(F)$ in the majority of cases.

Our present result is suggestive of similar applications: a bound for local Whittaker newforms (and a corresponding global sup-norm bound) in the level aspect; bounds for matrix coefficients of local representations, and estimates relating to the Vorono\u{\i} summation problem for $\GL_{n}$, to name a few.

In \S \ref{sec:intro-formula} we describe how irreducible, admissible representations of $\GL_{n}(F)$ are classified and then go on to give a full account of our main result. This classification assumes the least amount of necessary information in order to give a completely explicit formula. In \S \ref{sec:division} we give a uniform proof of our main result for the quasi-square-integrable representations (see Proposition \ref{prop:sq-int-formula}); these representations are used as building blocks to arrive at the general case. Lastly, in \S \ref{sec:analysis}, we provide a detailed analysis of the terms $\Delta_{\chi}(\pi)$ and $\delta_{\chi}(\pi)$ as found in \eqref{eq:intro-formula}.

\section{An explicit formula for twisted conductors}\label{sec:intro-formula}

Here we give full details of the formula proposed in \eqref{eq:intro-formula}. We first describe the formula for quasi-square-integrable representations of $\GL_{n}(F)$, which is then used to build the result in its full generality.

\subsection{The Langlands classification for \texorpdfstring{$\GL_{n}(F)$}{GL(n,F)}}

Let $\As_{F}(n)$ denote the set of (equivalence classes of) irreducible, admissible representations of $\GL_{n}(F)$. The natural building blocks that describe $\As_{F}(n)$ are the \textit{quasi-square-integrable} representations; these are the $\pi\in\As_{F}(n)$ for which there exists $\alpha\in \R$ such that $\abs{\ \cdot\ }^{\alpha}\pi$ has square-integrable matrix coefficients on $\GL_{n}(F)$ modulo its centre.

The `Langlands classification' (due to Berstein--Zelevinsky in this case) describes the structure of each representation in the graded ring $\As_{F}=\oplus_{n\geq 1} \As_{F}(n)$ in terms of the subset $\SG_{F}$ of quasi-square-integrable representations. By \cite[Theorems 9.3 \& 9.7]{zelevinsky}, one deduces an addition law $\boxplus$ on $\SG_{F}$, by which $\SG_{F}$ generates a free commutative monoid $\Lambda$. The classification is then the assertion that there is a bijection between $\As_{F}$ and the semi-group of non-identity elements in $\Lambda$, thus endowing $\As_{F}$ with the addition law $\boxplus$. Crucially, the maps $(\As_{F}, \boxplus\, )\rightarrow (\C,\ \cdot\ )$, given by applying $L$- or $\varepsilon$-factors, are homomorphisms of semi-groups (see \cite[\S 2.5]{wedhorn} for their definitions). Both expositions \cite{prasad-raghuram,wedhorn} provide excellent background on this topic.


The upshot of this classification is that for any $\pi\in \As_{F}(n)$ there exists a unique partition $n_{1}+\cdots+n_{r}=n$ alongside a collection of quasi-square-integrable representations $\pi_{i}\in \SG_{F}\cap\As_{F}(n_{i})$ for $1\leq i \leq r$ such that
\begin{equation}\label{eq:decomposition-pi}
\pi = \pi_{1}\boxplus \cdots \boxplus \pi_{r}
\end{equation}
and for any quasi-character $\chi$ of $\Fx$ we have
\begin{equation}\label{eq:decomposition-a-chi-pi}
a(\chi\pi) = a(\chi\pi_{1})+ \cdots + a(\chi\pi_{r}).
\end{equation}
Equation \eqref{eq:decomposition-a-chi-pi} follows from the definition of the conductor $a(\pi)$ via the $\varepsilon$-factor in \eqref{eq:epsilon-factor-formula}. Recall too that, for a quasi-character $\chi$ of $\Fx$, the conductor $a(\chi)$ is defined to be the least non-negative integer such that $\chi(\ofx\cap (1+\pf^{a(\chi)}))=\{1\}$ where $\of$ is the ring of integers of $F$ and $\pf\subset \of$ the unique maximal ideal.

\subsection{The formula for quasi-square-integrable representations}

\begin{defn}\label{def:twist-minimal}
An irreducible, admissible representation $\pi$ of $\GL_{n}(F)$ is called \textit{twist minimal} if $a(\pi)$ is the smallest of the integers $a(\chi\pi)$ as $\chi$ varies over the quasi-characters of $\Fx$.
\end{defn}

Recall that for a quasi-character $\chi$ of $\Fx$, define its conductor $a(\chi)$ to be the least non-negative integer such that $\chi(U_{F}(a(\chi)))=\{1\}$.
For quasi-square-integrable representations, the notion of twist-minimality is sufficient to give an exact formula for the conductor of their twist.


\begin{prop}\label{prop:sq-int-formula}
Let $\pi$ be an irreducible, admissible, quasi-square-integrable representation of $\GL_{n}(F)$ and let $\chi$ be a quasi-character of $\Fx$. Then
\begin{equation}\label{eq:sq-int-formula}
a(\chi\pi)\leq \max\{a(\pi),\,na(\chi)\}
\end{equation}
with equality in \eqref{eq:sq-int-formula} whenever $\pi$ is twist minimal or $a(\pi)\neq n a(\chi)$.
\end{prop}

We defer our proof of Proposition \ref{prop:sq-int-formula} until \S \ref{sec:main-proofs}.

\begin{rem}\label{rem:minimal}
In practice, one handles those $\pi\in\SG_{F}\cap\As_{F}(n)$ which are \textit{not} twist minimal as follows. Tautologically, write $\pi=\mu \pi^{\min}$ where $\mu$ is a quasi-character of $\Fx$ and $\pi^{\min}$ \textit{is} twist minimal. Then Proposition \ref{prop:sq-int-formula} implies that $a(\chi\pi)=\max\{a(\pi^{\min}),\,na(\chi\mu)\}$.
In particular, if $a(\pi^{\min})<a(\pi)$ then $n\mid a(\pi)$.
\end{rem}

Let us briefly mention the conductor formula of Bushnell--Henniart--Kutzko \cite[Theorem 6.5]{bushnell-henniart-kutzko} for $\GL_{n}\times\GL_{m}$-pairs of \textit{supercuspidal} representations. There they deploy the full structure theory of supercuspidal representations to prove a detailed identity relating the conductor to the respective inducing data of the given supercuspidal representations. However, this formula is difficult to apply in practice. Indeed, our own Proposition \ref{prop:sq-int-formula} may be derived from their work. Comparing the $m=1$ case of \cite{bushnell-henniart-kutzko} to our present result, our formula is simpler and holds uniformly on the larger set $\SG_{F}$. This set contains not only the supercuspidal representations but also, for example, the special representations, for which Proposition \ref{prop:sq-int-formula} recovers the known formula of Rohrlich \cite[p.~18]{rohrlich}. Accordingly, we give an elementary proof of Proposition \ref{prop:sq-int-formula}. This promotes our observation that the subset of twist minimal elements in $\SG_{F}$ contains sufficient and necessary information to explicitly determine the conductor of any twist.

The arguments of \S \ref{sec:main-proofs} also lead to a proof of the following result on the central character.

\begin{prop}\label{prop:cen-char}
Let $\pi$ be an irreducible, admissible, quasi-square-integrable representation of $\GL_{n}(F)$ with central character $\omega_{\pi}$. Then

\begin{equation}\label{eq:cen-char-formula}
a(\omega_{\pi})\,\leq\,  \frac{a(\pi)}{n}.
\end{equation}
\end{prop}

\begin{rem}
The central character of a quasi-square-integrable representation has relatively small conductor.  In general, highly ramified central characters arise due to the components in a given $\pi_{1}\boxplus\cdots\boxplus\pi_{r}$ for $r\geq 2$. For this reason, such representations should be handled separately, as is distinguished in this work.
\end{rem}

\subsection{The general formula}

We arrive at our main result, having defined the necessary set of properties of the representations in $\As_{F}$ in order to give a complete and explicit formula for the conductor of their twists.

\begin{theorem}\label{thm:exact-formula}

Let $\pi$ be an irreducible, admissible representation of $\GL_{n}(F)$ given in terms of quasi-square-integrable representations $\pi_{i}$ of $\GL_{n_{i}}(F)$, as described in \eqref{eq:decomposition-pi}, where $n=n_{1}+\cdots+n_{r}$ and $\pi=\pi_{1}\boxplus\cdots\boxplus\pi_{r}$. Let $\chi$ be a quasi-character of $\Fx$. Then
\begin{equation*}
a(\chi\pi)=a(\pi)+\Delta_{\chi}(\pi) - \delta_{\chi}(\pi)
\end{equation*}
where $\Delta_{\chi}$ and $\delta_{\chi}$ are semi-group homomorphisms $(\As_{F}, \boxplus\, )\rightarrow (\Z_{\geq 0}, + )$ defined by their values on the representations $\pi_{i}\in \SG_{F}$ as follows:
\begin{equation*}
\Delta_{\chi}(\pi_{i})=\left\lbrace\begin{array}{cl}\vspace{0.05in}
\max\{ n_{i}a(\chi)-a(\pi_{i}),\,0\} & \text{if } a(\chi)\neq a(\mu_{i})\\
0&\text{if } a(\chi)= a(\mu_{i})
\end{array}\right.
\end{equation*}
and
\begin{equation*}
\delta_{\chi}(\pi_{i})=\left\lbrace\begin{array}{cl}\vspace{0.05in}
 a(\pi_{i})-\max\{a(\pi^{\min}_{i}),\,n_{i}a(\chi\mu_{i})\} & \text{if } a(\chi)= a(\mu_{i})\\\vspace{0.05in}
0&\text{if } a(\chi)\neq a(\mu_{i})
\end{array}\right.
\end{equation*}
where $\pi^{\min}_{i}$ is twist minimal and $\mu_{i}$ a quasi-character of $\Fx$ such that we may write $\pi_{i}=\mu_{i}\pi^{\min}_{i}$.
\end{theorem}

\begin{rem}\label{rem:non-negative}
As exhibited in the following proof, both terms $\Delta_{\chi}(\pi)$ and $\delta_{\chi}(\pi)$ are non-negative for any choice of $\pi$ and $\chi$.
\end{rem}

\begin{proof}

Applying Proposition \ref{prop:sq-int-formula} to the formula in \eqref{eq:decomposition-a-chi-pi} we obtain
\begin{equation}\label{eq:proof-key}
a(\chi\pi)=\sum_{i=1}^{r} \max\{a(\pi^{\min}_{i}),\,n_{i}a(\chi\mu_{i})\}.
\end{equation}
We now use the basic fact that for two quasi-characters, $\mu$ and $\chi$, of $\Fx$ we have
\begin{equation}\label{eq:char-inequality}
a(\chi\mu) \leq \max\{a(\chi),a(\mu)\}
\end{equation}
with equality in \eqref{eq:char-inequality} whenever $a(\chi)\neq a(\mu)$. In particular, if $a(\chi)\neq a(\mu_{i})$ for a given $1\leq i\leq r$ then, by Proposition \ref{prop:sq-int-formula} and \eqref{eq:char-inequality}, the respective summand in \eqref{eq:proof-key} is equal to $\max\{a(\pi^{\min}_{i}),\,n_{i}a(\chi\mu_{i})\}=\max\{a(\pi_{i}),\,n_{i}a(\chi)\}$.
This determines the dominant term $\Delta_{\chi}(\pi_i)$, which is non-negative by construction. The interference term $\delta_{\chi}(\pi_i)$ describes the cases for which $a(\chi)= a(\mu_{i})$, when the assertion that $\delta_{\chi}(\pi_i)\geq 0$ follows from the inequality $ a(\pi_{i})\geq\max\{a(\pi^{\min}_{i}),\,n_{i}a(\chi\mu_{i})\}$.
\end{proof}

\begin{rem}
In the special case $n=2$, we prove Theorem \ref{thm:exact-formula} in \cite[Lemma 2.7]{corbett-saha}. In general, one should understand the non-vanishing of $\delta_{\chi}(\pi)$ as occurring rarely, whereas $\Delta_{\chi}(\pi)$ describes the dominant or ``usual'' behaviour of $a(\chi\pi)$. We make these statements explicit in a quantitative sense in \S \ref{sec:leading-interference}.
\end{rem}

\begin{cor}\label{cor:upper-and-lower-bounds}

Let $\pi=\pi_{1}\boxplus\cdots\boxplus\pi_{r}$ and $\chi$ be as in Theorem \ref{thm:exact-formula} with $\pi_{i}=\mu_{i}\pi^{\min}_{i}$ for twist minimal representations $\pi_{i}^{\min}$. Define the `totally minimal' representation $\pi^{\tot}=\pi^{\min}_{1}\boxplus \cdots \boxplus \pi^{\min}_{r}$ and let $\Omega_{\chi}(\pi)=\{ 1\leq i \leq r \, : \, a(\pi_{i})> n_{i}a(\chi)\}$. Then
\begin{equation}\label{eq:cor-bounds}
a(\pi^{\tot})\leq a(\chi\pi)\leq a(\pi) +a(\chi)\bigg(n-\sum_{i\in \Omega_{\chi}(\pi)}n_{i}\bigg).
\end{equation}

\end{cor}
\begin{proof}

The lower bound of \eqref{eq:cor-bounds} follows immediately from \eqref{eq:decomposition-a-chi-pi} and \eqref{eq:proof-key}. On the other hand, for $i\in \Omega_{\chi}(\pi)$ we have $\Delta_{\chi}(\pi_{i})=\delta_{\chi}(\pi_{i})=0$ by definition, noting that $\pi_i=\pi_{i}^{\min}$ in the case $a(\chi)=a(\mu_i)$. The upper bound now follows from using Proposition \ref{prop:sq-int-formula} to coarsely bound $a(\chi\pi_i)\leq a(\pi_{i})+n_{i}a(\chi)$ for $i\not\in \Omega_{\chi}(\pi)$.

\end{proof}

\begin{proof}[Proof of Inequality \eqref{eq:bushnell-bound}]
We recover Bushnell--Henniart's bound \eqref{eq:bushnell-bound} using Corollary \ref{cor:upper-and-lower-bounds}. If $a(\chi)>a(\pi)$ then $a(\chi\pi)=na(\chi)$ by \eqref{eq:proof-key}. On the other hand, if $a(\chi)\leq a(\pi)$ then \eqref{eq:bushnell-bound} is a special case of \eqref{eq:cor-bounds} since we have $\Omega_{\chi}(\pi)\neq \emptyset$ and each $n_{i}\geq 1$.
\end{proof}

\section{Conductors of twists via division algebras}\label{sec:division}

In this section we provide proofs for Propositions \ref{prop:sq-int-formula} and \ref{prop:cen-char}. These results apply to all quasi-square-integrable representations uniformly, as is reflected in our proof. In particular, our conductor formula bypasses many of the complications occuring in the formula for supercuspidal representations given in \cite{bushnell-henniart-kutzko}.

\subsection{Notation and definition of the conductor}\label{sec:preliminar-notions}

Let $\pi$ denote an irreducible, admissible representation of $\GL_{n}(F)$. Denote by $\tilde{\pi}$ be the contragredient representation and $\omega_{\pi}$ the central character of $\pi$, respectively.

\subsubsection{The non-archimedean local field}

We denote by $\of$ the ring of integers of $F$; $\pf$ the maximal ideal of $\of$; $\varpi$ a choice of uniformising parameter, that is a generator of $\pf$; and $q=\# (\of/\pf)$. Let $\abs{x}$ denote the absolute value of $x\in F$, normalised so that $\abs{\varpi}=q^{-1}$, and $v_{F}$ the valuation on $F$ defined via $\abs{x} = q^{-v_{F}(x)}$. We define a basis of open neighbourhoods $U_{F}(m)$ of $1$ in $U_{F}(0)=\ofx$ by $U_{F}(m)=1+\varpi^{m}\of$ for $m>0$. Let $K=\GL_{n}(\of)$ and for each $m\geq 0$ let $K_{1}(m)$ be the subgroup of $K$ stabilising the row vector $(0,\ldots,0,1)$, from the right, modulo $\pf^{m}$.

\subsubsection{The floor and ceiling functions}

For $\alpha\in\R$ let $\lfloor \alpha\rfloor$ denote the \textit{floor} of $\alpha$, defined via $\lfloor \alpha\rfloor=m$ if and only if $m\in\Z$ and $m\leq \alpha < m+1$. Similarly, let $\lceil \alpha\rceil$ denote the \textit{ceiling} of $\alpha$, defined via $\lceil \alpha\rceil=m'$ if and only if $m'\in\Z$ and $m'-1< \alpha\leq m'$. Then $\lfloor \alpha\rfloor=\lceil \alpha\rceil$ if and only if $\alpha\in\Z$.

\subsubsection{Epsilon constants and the conductor}\label{sec:conductor-artin}

Here we define the integer $a(\pi)$, the conductor of $\pi$. 
Let $\psi$ be an additive character of $F$ and define the exponent of $\psi$ by $n(\psi):=\min \{m : \psi\vert_{\pf^{m}}=1\}.$ Godement--Jacquet prove the existence of $\varepsilon$-factors $\varepsilon(s,\pi,\psi)\in \C[q^{-s},q^{s}]$ in \cite[Theorem 3.3, (4)]{godement-jacquet}. Applying the local functional equation of Godement--Jacquet twice, one obtains
\begin{equation}
\varepsilon(s,\pi,\psi)\varepsilon(1-s,\tilde{\pi},\psi)=\omega_{\pi}(-1).
\end{equation}
Hence $\varepsilon(s,\pi,\psi)$ is a unit in $\C[q^{-s},q^{s}]$; that is, a $\Cx$-constant multiple of an integral power of $q^{-s}$. Explicitly, using \cite[(3.3.5)]{godement-jacquet} one deduces
\begin{equation}\label{eq:epsilon-factor-formula}
\varepsilon(s,\pi,\psi)=\varepsilon(1/2,\pi,\psi)\,q^{(a(\pi)-n(\psi)n)(\frac{1}{2}-s)},
\end{equation}
in which the conductor $a(\pi)$ is implicitly defined. By the local Langlands correspondence for $\GL_{n}(F)$, proved in \cite{harris-taylor-langlands}, the conductor $a(\pi)$ coincides with the Artin conductor of an $n$-dimensional Weil--Deligne representation. A fundamental property of $\varepsilon$-factors is that $\varepsilon(s,\chi\pi,\psi)=\prod_{i=1}^{r}\varepsilon(s,\chi\pi_{i},\psi)$ for $\pi=\pi_{1}\boxplus\cdots\boxplus\pi_{r}$, as in \eqref{eq:decomposition-pi} (see \cite[Theorem 3.4]{godement-jacquet}). This observation proves \eqref{eq:decomposition-a-chi-pi} by applying \eqref{eq:epsilon-factor-formula}. Moreover, if $\pi$ is generic, the conductor $a(\pi)$ may be interpreted in terms of newform theory as we now explain.

\subsubsection{Conductors of generic representations and newform theory}\label{sec:conductor-newform}

Each representation in $\SG_{F}$ is \textit{generic}. Indeed, by showing so for the regular representation of $\GL_{n}(F)$ of fixed central character, Jacquet shows that all discrete series representations are generic \cite[Theorem 2.1, (3)]{jacquet-generic}. By the Langlands classification, any $\pi\in\As_{F}(n)$ is generic (or ``non-degenerate'') if and only if $\pi$ is equivalent to the (irreducible) representation parabolically induced from the external tensor product $\pi_{1}\boxtimes\cdots\boxtimes\pi_{r}$ of $\GL_{n_1}(F)\times\cdots\times\GL_{n_r}(F)$ associated to $n_{1}+\cdots+n_{r}$ (by \cite[Theorem 9.7, (a)]{zelevinsky}). The elements of $\SG_{F}$ correspond to those irreducible representations with $r=1$.

Assume that $\pi\in\As_{F}(n)$ is generic. Then the conductor $a(\pi)$ may be equivalently constructed in a language more familiar to the theory of automorphic forms: let us re-define the conductor $a(\pi)$ of $\pi$ to be the least non-negative integer $m$ such that $\pi$ contains a non-zero $K_{1}(m)$-fixed vector.

The fundamental theorem of newform theory is that the space of $K_{1}(a(\pi))$-fixed vectors is one-dimensional. This theorem is due to Gelfand--Ka\v{z}dan \cite{gelfand-kazdan} in the present context. The coincidence of the definitions for $a(\pi)$ given in \S \ref{sec:conductor-artin} and \S \ref{sec:conductor-newform} is proved by Jacquet--Piatetski-Shapiro--Shalika \cite[Th{\'e}or{\`e}me (5)]{jpss}.

\subsection{Central simple division algebras}


Let $D$ be a division algebra over $F$ of dimension $[D:F]=n^{2}$. Let $\Nrd=\Nrd_{D}$ denote the reduced norm on $D$. (See \cite[\S 4.1]{koch-zink} for a pleasant construction.) Any valuation on $D$ may be obtained via composing the reduced norm with a valuation on $F$ (see \cite[Theorem 1.4]{tignol-wadsworth}); let us normalise such a choice by $v_{D}=v_{F}\circ\Nrd$.

\subsubsection{Unit groups}
Define a basis of neighbourhoods of $1\in \Dx$ by $U_{D}(m)=\lbrace x\in \Dx \,:\, v_{D}(x-1)\geq m\rbrace$ for $m>0$ and let $U_{D}(0)=\ker(v_{D})$. Note that if $n=1$ (so that $D=F$) we recover $U_{D}(m)=U_{F}(m)$. It is an important fact that the norm map $\Nrd\colon \Dx \rightarrow\Fx$ is surjective (see \cite[Prop.\ 6, Ch.\ X-2, p.\ 195]{weil-basic} for instance). Upon restriction to the above neighbourhoods, for each $m\geq 0$ we have $\Nrd(U_{D}(m))=U_{D}(m)\cap F$.

\begin{lem}\label{lem:div-intersection-of-units}
For $m\geq 0$ we have the following:
\begin{enumerate}

\item $U_{D}(m)\cap \Fx=U_{F}(\lceil m/n\rceil )$;

\item $\Nrd(U_{D}(m))=U_{F}(\lceil m/n\rceil )$.

\end{enumerate}
\end{lem}

\begin{proof}
To prove (1), note that for all $a\in\Fx$ we have $v_{D}(a)=v_{F}(\Nrd(a))=v_{F}(a^{n})=nv_{F}(a)$. The definition of $U_{F}(\lceil m/n\rceil )$ is then equivalent to that of the intersection. Now (2) follows by applying (1) to $\Nrd(U_{D}(m))=U_{D}(m)\cap F$.
\end{proof}




\subsubsection{The level of a representation of \texorpdfstring{$\Dx$}{Dx}}

If $\chi$ is a quasi-character of $\Fx$ and $\pi'$ an irreducible, admissible representation of $\Dx$, analogous to the unramified case we form the twist $\chi\pi'=(\chi\circ \Nrd )\otimes \pi'$. Define the \textit{level} $l(\pi')$ of $\pi'$ to be the least non-negative integer $m$ such that $\pi'\vert_{U_{D}(m)}$ acts trivially. The notion of and $\varepsilon$-factor, as well as conductor $a(\pi')$, is defined by Godement--Jacquet \cite{godement-jacquet}, mutatis mutandis as in \S \ref{sec:conductor-artin}.

\begin{lem}\label{lem:div-level-koch}
Let $\pi'$ be an irreducible, admissible representation of $\Dx$. The conductor $a(\pi')$ is related to the level $l(\pi)$ by the formula
\begin{equation*}
a(\pi')=l(\pi')+n-1.
\end{equation*}

\end{lem}

\begin{proof}
This is proved in \cite[\S 4.3]{koch-zink} and stated explicitly in \cite[(4.3.4)]{koch-zink}. To assist with (mathematical) translation, we remark on the following: their unit groups $V_{j}$ equal our $U_{D}(j)$ for $j\geq 0$. Fix their element $\chi\in\Hom(V_{j}/V_{j+1},\Cx)$ to be the restriction of $\pi'$ to $V_{j}$ where $j=l(\pi')-1$. Then their $c\in D$, ``der Kontrolleur von $\chi$'', satisfies $v_{D}(c)=-a(\chi)=-a(\pi')$; it is constructed in \cite[(4.3.1)]{koch-zink} from where we have $v_{D}(c)=-n-j$, noting the non-triviality of $\chi$ on $V_{j}$. All together this implies $a(\pi')=n+j=n+l(\pi')-1$.
\end{proof}

\begin{lem}\label{lem:div-level-of-chi}
Let $\chi$ be a quasi-character of $\Fx$. Then
\begin{equation*}
l(\chi\circ\Nrd)=na(\chi)-n+1.
\end{equation*}
\end{lem}

\begin{proof}

By Lemma \ref{lem:div-intersection-of-units}, (2) consider $\chi$ restricted to $U_{F}(\lceil m/n \rceil)$ for each $m\geq 0$ as this set is equal to the image of $U_{D}(m)$ under $\Nrd$. By the minimality of $a(\chi)$, the character $\chi\circ\Nrd$ is trivial on $U_{D}(m)$ whenever
\begin{equation}\label{eq:div-min-level}
n(a(\chi)-1)\leq m -1.
\end{equation}
By the minimality of the level, we have equality in \eqref{eq:div-min-level} when $m=l(\chi\circ\Nrd)$.
\end{proof}

\subsection{The Jacquet--Langlands correspondence for division algebras}\label{sec:jacquet-langlands}

This special case of functoriality stipulates a bijection between the following:

\begin{itemize}

\item The set of equivalence classes of irreducible, admissible representations of $\GL_{n}(F)$, with unitary central character, which are \textit{square-integrable modulo centre}. These are precisely the square-integrable elements of $ \SG_{F}\cap \As_{F}(n)$.

\item The set of equivalence classes of irreducible, admissible representations of $\Dx$ with unitary central character where $D$ is a central-simple $F$-algebra of dimension $n^{2}$.

\end{itemize}

\begin{rem}
In the above bijection, if $\pi$ corresponds to $\pi'$ then their central characters agree: $\omega_{\pi}=\omega_{\pi'}$. Moreover, $\chi\pi$ corresponds to $\chi\pi'$ for any quasi-character $\chi$. As a consequence of the Peter--Weyl theorem, the irreducible representations of $\Dx$ are finite dimensional (since $\Dx$ is compact modulo centre).
\end{rem}

The correspondence as stated here is due to Rogawski \cite[Theorem 5.8]{rogawski}, where the original case $n=2$ was famously proved by Jacquet--Langlands \cite{jacquet-langlands}. The most general statement allows one to replace $\Dx$ with $\GL_{m}(D)$ where $D$ has dimension $d^{2}$ and $m$ must satisfy $n=md$. This is established in \cite{dkv} by Deligne--Kazhdan--Vign\'eras.

\subsection{The main proofs}\label{sec:main-proofs}

Here we provide a stand-alone proof of Proposition \ref{prop:sq-int-formula}, our main result in the quasi-square-integrable case. Assume the hypotheses and notations of Propositions \ref{prop:sq-int-formula} and \ref{prop:cen-char}; in particular, $\pi\in\SG_{F}$.

\subsubsection{Proof of Proposition \ref{prop:sq-int-formula}}

The following lemma reduces the proof to the case where $\pi$ is square-integrable.

\begin{lem}
For all quasi-characters $\chi$ with $a(\chi)=0$ we have $a(\chi\pi)=a(\pi)$.
\end{lem}

\begin{proof}
Let $m\geq 0$. The space $\pi^{K_{1}(m)}$ of $K_{1}(m)$-fixed vectors in $\pi$ is non-zero if and only if $(\chi\pi)^{K_{1}(m)}\neq \{0\}$. As $\pi\in \SG_{F}$, both $\pi$ and $\chi\pi$ are generic, and so $a(\pi)=\min\{m\geq 0 : \pi^{K_{1}(m)}\neq 0\}=a(\chi\pi)$.
\end{proof}

Henceforth we assume $\pi$ to be square-integrable. The generalised Jacquet--Langlands correspondence implies $a(\chi\pi)=a(\chi\pi')$ where $\pi'$ is the irreducible, admissible, unitary representation of $\Dx$ associated to $\pi$ as determined by \cite[Theorem 5.8]{rogawski}. The proof of Proposition \ref{prop:sq-int-formula} now follows by applying Lemmas \ref{lem:div-level-koch} and \ref{lem:div-level-of-chi} to the following.

\begin{lem}
Let $\pi'$ be an irreducible, admissible, unitary representation of $\Dx$ and $\chi$ a quasi-character of $\Fx$. Then

\begin{equation}\label{eq:div-level-inequality}
l(\chi\pi')\leq \max\{l(\pi'),\,l(\chi\circ\Nrd)\}
\end{equation}
with equality in \eqref{eq:div-level-inequality} whenever $\pi'$ is twist minimal or $l(\pi')\neq l(\chi\circ\Nrd)$.
\end{lem}

\begin{proof}
By definition, $(\chi\pi')(x)=\chi(\Nrd(x))\pi'(x)$ for every $x\in \Dx$. One immediately obtains \eqref{eq:div-level-inequality} by minimality. Equality also follows in the given cases, noting that twist minimality in $a(\pi')$ is equivalent to twist minimality in $l(\pi')$ since they are linearly related (by Lemma \ref{lem:div-level-koch}).

\end{proof}

\subsubsection{Proof of Proposition \ref{prop:cen-char}}

Taking $m=l(\pi')$ in Lemma \ref{lem:div-intersection-of-units}, (1) and using the formula of Lemma \ref{lem:div-level-koch}, we deduce that
\begin{equation*}
a(\omega_{\pi})\leq \left\lceil\frac{l(\pi')}{n}\right\rceil<\frac{a(\pi)-n+1}{n}+1=\frac{a(\pi)+1}{n}.
\end{equation*}
Thus we infer that $na(\omega_{\pi})\leq a(\pi)$ as required.

\section{Characters preserving the conductor under twisting}\label{sec:analysis}

The goal of this section is twofold: in \S \ref{sec:spaces-of-characters} we count the number of characters $\chi$ such that $a(\chi\pi)$ is equal to a given integer. Then, in \S \ref{sec:leading-interference}, we explicitly analyse the behaviour of the dominant and interference terms of Theorem \ref{thm:exact-formula}. These questions are motivated by their applications to analytic number theory.

\subsection{Sets of twist-fixing characters}\label{sec:spaces-of-characters}

\subsubsection{Characters of a given conductor}

The valuation $v_{F}$ defines a split exact sequence $1\longrightarrow \ofx \longrightarrow \Fx \displaystyle{\arrup{\longrightarrow}{v_{F}}} \Z\longrightarrow 1$. We thus write any quasi-character $\chi$ on $\Fx$ as $\chi(x)=\chi'(x)q^{-v_{F}(x)\alpha}$ for some $\alpha\in \C$ and a character $\chi'$ of $\Fx$ such that $\chi'(\varpi)=1$. We denote the space of such $\chi'$ by $\Xf$ so that the unitary dual of $\ofx$ satisfies $\hat{\of}^{\times}\isom \Xf$. With interest in characters that fix the conductor under twisting, we define the following $\Xf$-subsets:
\begin{equation}
\Xf(k)=\{ \chi \in \Xf \, : \, a(\chi)\leq k\};\quad\Xf'(k)=\{ \chi \in \Xf \, : \, a(\chi) = k\};
\end{equation}
and
\begin{equation}
\Xf_{\pi}'(k,j)=\{ \chi \in \Xf \, : \, a(\chi) = k \text{ and }a(\chi\pi) = j \}
\end{equation}
for some $k,j\geq 0$.

Our present point of departure is to count the number of characters contained in $\Xf_{\pi}'(k,j)$. We first consider the cardinalities of $\Xf(k)$ and $\Xf'(k)$.

\begin{lem}\label{lem:number-of-chars}
For each $k \geq 1$, $\#\Xf(k)=q^{k-1}(q-1)$, $\#\Xf'(1) =  q-2$, and for $k\geq 2$, $\#\Xf'(k) =q^{k-2}(q-1)^{2}$.
\end{lem}

\begin{proof}
Consider the subgroup series $\{1\}=\Xf(0)\leq \Xf(1)\leq\cdots \leq\Xf(k)\leq\Xf$. For $k \geq l \ge k/2 \geq 1$, we have $\Xf(k) / \Xf(l) \isom U_{F}(l)/U_{F}(k) \isom \of/\pf^{k-l}$. In particular, taking $l=k-1$ and noting $\Xf(1)\isom(\of/\pf)^{\times}$, one counts the given cardinalities inductively. The number $\#\Xf'$ is obtained by subtraction.
\end{proof}

We remark that in \cite[Lemmas 2.1 \& 2.2]{corbett-saha} we counted the elements $\chi\in\Xf'(k)$ for which $a(\chi\mu)$ remains fixed for a given $\mu\in\Xf'(k)$, characterising the existence of such elements as $q$ becomes small. In the present work we consider a ``nonabelian'' variant of this result by characterising the set $\Xf_{\pi}'(k,j)$.

\subsubsection{Character twists of a given conductor}

Suppose that $\pi\in \SG_{F}\cap\As_{F}(n)$ so that Proposition \ref{prop:sq-int-formula} applies. For integers $k,j\geq 0$, if either $\pi$ is twist minimal or $k\neq a(\pi)/n$ then
\begin{equation}\label{eq:trivial-cases}
\Xf_{\pi}'(k,j)=\left\lbrace\begin{array}{cl}\vspace{0.05in}
\Xf_{\pi}'(k)& \text{if }j= \max\{a(\pi),\,nk\}\\\vspace{0.05in}
\emptyset & \text{if }j\neq \max\{a(\pi),\,nk\}.
\end{array}\right.
\end{equation}
The cases considered in \eqref{eq:trivial-cases} are special cases of the following lemma.

\begin{lem}\label{lem:twist-fixing-sq}
For each $\pi\in\SG_{F}\cap\As_{F}(n)$ write $\pi=\mu\pi^{\min}$ for a twist minimal representation $\pi^{\min}$. For integers $j,k\geq 0$ we have $\Xf_{\pi}'(k,j)=\emptyset$ unless $a(\pi^{\min})\leq j \leq  \max\{a(\pi),\, nk\}$, in which case
\begin{equation}
\#\Xf_{\pi}'(k,j)\leq \#\Xf\big(\big\lfloor\tfrac{j}{n}\big\rfloor\big).
\end{equation}
\end{lem}
\begin{proof}
If either $\pi$ is minimal or $k\neq a(\pi)/n$ then the lemma follows by \eqref{eq:trivial-cases}. Hence assume $a(\pi)=kn$ and $\pi=\mu\pi^{\min}$ where $\pi^{\min}$ is twist minimal with $a(\pi^{\min})<a(\pi)$ and $\mu\in\Xf'(k)$. Then $\Xf'_{\pi}(k,j)=\emptyset$ unless $a(\pi^{\min})\leq j \leq nk$. In this case, if there exists a $\chi\in \Xf'(k)$ such that $\max\{a(\pi^{\min}),\,na(\chi\mu)\}=j$ then there are $\#\Xf(\lfloor j/n\rfloor)$ of them as we must have $\chi\in \mu^{-1}\Xf(\lfloor j/n\rfloor)$.

\end{proof}

More generally, Lemma \ref{lem:twist-fixing-sq} may be assembled to describe all of $\As_{F}(n)$.

\begin{cor}\label{cor:twist-fix}
Let $\pi\in\As_{F}(n)$. For integers $j,k\geq 0$ we have $\Xf_{\pi}'(k,j)=\emptyset$ if $j> a(\pi) + nk$. Write $\pi=\pi_1\boxplus\cdots\boxplus\pi_r$ as in \eqref{eq:decomposition-pi}.
\begin{enumerate}
\item For each $1\leq i\leq r$, if $\pi_i$ is either minimal or $a(\pi_i)\neq kn_i$ then $$\#\Xf_{\pi}'(k,j)\leq \#\Xf\big(\big\lfloor\tfrac{j}{n}\big\rfloor\big).$$ 
\item Otherwise, define the set of indices $\Psi_{k}(\pi)\subset \{1,\ldots,r\}$ such that $i\in \Psi_{k}(\pi)$ if and only if $ a(\pi_{i})=n_{i}k$ and $a(\pi_{i}^{\min})<a(\pi_{i})$, where $\pi_{i}^{\min}$ is a minimal representation satisfying $\pi_{i}=\mu_{i}\pi_{i}^{\min}$. Then for any $i'\in\Psi_{k}(\pi)$ we have $$\#\Xf_{\pi}'(k,j)\leq \#\Xf\big(\bigg\lfloor\frac{1}{n_{i'}}\big(
j-\sum_{i\not\in\Psi_{k}(\pi)}n_i k -\sum_{i\in\Psi_{k}(\pi) \sm\{ i'\}}a(\pi_{i}^{\min})\big)
\bigg\rfloor\big).$$
\end{enumerate}

\end{cor}

\begin{proof}
The bound $j> a(\pi) + nk$ is derived from the fact that $a(\chi\pi)\leq a(\pi) + na(\chi)$ for any $\chi\in\Xf$ (see Corollary \ref{cor:upper-and-lower-bounds}). Now suppose $a(\chi)=k$ and $a(\chi\pi)=j$.
By Proposition \ref{prop:sq-int-formula} we have $a(\chi\pi_i)=\max\{a(\pi_i),n_i k\}$ for all $i\not\in \Psi_{k}(\pi)$. In particular, for $\chi\in\Xf'(k)$ we have that $ a(\chi\pi)=j$ if and only if
\begin{equation}\label{eq:chis-and-psis}
j=\sum_{i\in \Psi_{k}(\pi)}a(\chi \pi_i)+ \sum_{i\not\in  \Psi_{k}(\pi)}\max\{ a(\pi_i),n_i k\}.
\end{equation}
Then, if $\Psi_{k}(\pi)=\emptyset$ for each $\chi\in\Xf'(k)$, as in case (1), we have $a(\chi\pi)=j$ for each $\chi$, given \eqref{eq:chis-and-psis} holds. Moreover, since $j\geq kn$ we obtain $\#\Xf_{\pi}'(k,j)\leq\#\Xf(\lfloor j/n\rfloor)$ as claimed. Otherwise, pick $i'\in\Psi_{k}(\pi)$ as in case (2). If $a(\chi\pi)=j$ then
$a(\chi\pi_{i'})=j'$ where we define
$$j'= j- \sum_{i\neq i'}a(\chi\pi_{i}).$$
Then the number of $\chi\in \Xf'(k)$ such that $a(\chi \pi_{i'})=j'$ is at most $\#\Xf(\lfloor j'/n_{i'}\rfloor)$ by Lemma \ref{lem:twist-fixing-sq}, whence we deduce the claim.



\end{proof}


\subsection{The leading and interference terms}\label{sec:leading-interference}

Here we detail the asymptotic behaviour of $\Delta_{\chi}(\pi)$ and $\delta_{\chi}(\pi)$. Our first port of call is to describe the rarity with which the interference term satisfies $\delta_{\chi}(\pi)\neq 0$. The following lemma follows directly from the definition of $\delta_{\chi}(\pi)$ in Theorem \ref{thm:exact-formula}.

\begin{lem}[Absence of interference]\label{lem:interference}

Let $\pi$ be an irreducible, admissible representation of $\GL_{n}(F)$ written, as in \eqref{eq:decomposition-pi}, in terms of irreducible, quasi-square-integrable representations, $\pi=\pi_{1}\boxplus\cdots\boxplus\pi_{r}$. Recall that $\pi_{i}\in\SG_{F}$ is a representation of $\GL_{n_{i}}(F)$ for $1\leq i \leq r$. Let $\chi$ be a quasi-character of $\Fx$.

\begin{enumerate}

\item We have $\delta_{\chi}(\pi)=0$ if $n_{i}\nmid a(\pi_{i})$ for each $1\leq i \leq r$.

\item Suppose $n_{i}\mid a(\pi_i)$ for some $1\leq i \leq r$. Then $\delta_{\chi}(\pi_{i})=0$ whenever $a(\pi_{i})\neq n_{i}a(\chi)$.

\item Suppose $a(\pi_{i})= n_{i}a(\chi)$ for some $1\leq i \leq r$. Then $\delta_{\chi}(\pi_{i})= 0$ if and only if $a(\chi\mu_{i})=a(\chi)$ where $\pi_{i}=\mu_{i}\pi_{i}^{\min}$ is written as the $\mu_{i}$-twist of a minimal representation $\pi_{i}^{\min}$.

\end{enumerate}

\end{lem}

\begin{proof}
Recall that $\delta_{\chi}(\pi_i)=a(\pi_i)-\max\{a(\pi_{i}^{\min}),n_ia(\chi\mu_{i})\}$ for $a(\chi)=a(\mu_i)$ and vanishes otherwise. If $n_i\nmid a(\pi_i)$ then $a(\pi_i)=a(\pi_{i}^{\min})\geq n_ia(\mu_{i})\geq n_ia(\chi\mu_{i})$ for each $1\leq i\leq r$. This proves (1). For (2) we let $n_i\mid a(\pi)$. If $a(\pi^{\min}_{i})=a(\pi_i)$ we argue as in (1). Else, $a(\pi_i)=n_{i}a(\mu_i)=n_{i} a(\chi)$ when $\delta_{\chi}(\pi_i)\neq 0$, as claimed. The vanishing of $\delta_{\chi}(\pi_i)$ in (3) is characterised by the condition $n_{i}a(\chi)=\max\{a(\pi_{i}^{\min}),n_{i}a(\chi\mu_{i})\}$ for $a(\chi)=a(\mu_i)$. If $n_{i}a(\chi)=a(\pi_{i}^{\min})$ we again argue as in (1), forcing the remaining condition $a(\chi\mu_{i})=a(\chi)$.
\end{proof}

\begin{cor}[Dominant behaviour]
In each case of Lemma \ref{lem:interference} for which $\chi$ and $\pi=\pi_{1}\boxplus\cdots\boxplus\pi_{r}$ satisfy $\delta_{\chi}(\pi)=0$, we have the ``dominant'' conductor formula
\begin{equation}
a(\chi\pi)=\sum_{i=1}^{r} \max \{ a(\pi_{i}) ,\,n_{i} a(\chi)\}.
\end{equation}
\end{cor}

Our final port of call is to quantify the rarity of $\delta_{\chi}(\pi)=0$, as in Lemma \ref{lem:interference}.

\begin{lem}[Regularity of interference]
Let $\pi=\pi_{1}\boxplus\cdots\boxplus \pi_{r}$ as in \eqref{eq:decomposition-pi}. Suppose $\chi\in\Xf$ and that for some $1\leq i \leq r$ we have $\delta_{\chi}(\pi_{i})\neq 0$. Write $\pi_{i}=\mu_{i}\pi_{i}^{\min}$ as per Lemma \ref{lem:interference}, (3). Then, for each $0< j\leq a(\pi_{i})-a(\pi_{i}^{\min})$ satisfying $j\equiv a(\pi_{i})\Mod{n_{i}}$, there are precisely
\begin{equation}\label{eq:quant-lemma-1}
\#\Xf\bigg(\frac{a(\pi_{i})-j}{n}\bigg)
\end{equation}
characters $\chi\in\Xf$ such that $\delta_{\chi}(\pi_{i})=a(\pi_{i})-j$. The number of $\chi\in \Xf(a(\pi_{i})/n)$ satisfying $\delta_{\chi}(\pi_{i})=a(\pi_{i})$ is 
\begin{equation}\label{eq:quant-lemma-2}
(q-2)\times\#\Xf\bigg(\frac{a(\pi_{i})}{n}-1\bigg).
\end{equation}
\end{lem}

\begin{proof}

The number in \eqref{eq:quant-lemma-1} is determined by the necessity that
\begin{equation*}
\chi\in\mu^{-1}_{i} \Xf\bigg(\frac{a(\pi_{i})-j}{n}\bigg).
\end{equation*}
Similarly, we count upto the number in \eqref{eq:quant-lemma-2} by observing that $\chi\in \Xf(a(\pi_{i})/n_{i})$ but $\chi$ is not an element of $\Xf((a(\pi_{i})/n_{i})-1)$ nor $\mu_{i}^{-1}\Xf((a(\pi_{i})/n_{i})-1)$.

\end{proof}

\section*{Acknowledgement}

The author would like to express gratitude to the Max-Planck-Institut f\"ur Mathematik, Bonn for providing welcoming hospitality during a visit to the institute during which this work was completed.


\bibliographystyle{amsplain}			
\bibliography{bibliography-cond}				

\providecommand{\bysame}{\leavevmode\hbox to3em{\hrulefill}\thinspace}
\providecommand{\MR}{\relax\ifhmode\unskip\space\fi MR }
\providecommand{\MRhref}[2]{%
  \href{http://www.ams.org/mathscinet-getitem?mr=#1}{#2}
}
\providecommand{\href}[2]{#2}
\begin{thebibliography}{10}

\bibitem{brunault}
F.~Brunault, \emph{On the ramification of modular parametrizations at the
  cusps}, J. Th\'eor. Nombres Bordeaux \textbf{28} (2016), no.~3, 773--790.

\bibitem{bushnell-henniart-upper-bound}
C.~Bushnell and G.~Henniart, \emph{An upper bound on conductors for pairs}, J.
  Number Theory \textbf{65} (1997), no.~2, 183--196.

\bibitem{bushnell-henniart-kutzko}
C.~Bushnell, G.~Henniart, and P.~Kutzko, \emph{Local {R}ankin-{S}elberg
  convolutions for {$\operatorname{GL}_n$}: explicit conductor formula}, J.
  Amer. Math. Soc. \textbf{11} (1998), no.~3, 703--730.

\bibitem{corbett-saha}
A.~{Corbett} and A.~{Saha}, \emph{On the order of vanishing of newforms at
  cusps}, Math. Res. Lett. (2018), To Appear.

\bibitem{dkv}
P.~Deligne, D.~Kazhdan, and M.-F. Vign\'eras, \emph{Repr\'esentations des
  alg\`ebres centrales simples {$p$}-adiques}, Representations of reductive
  groups over a local field, Travaux en Cours, Hermann, Paris, 1984,
  pp.~33--117.

\bibitem{gelfand-kazdan}
I.~Gelfand and D.~Ka\v{z}dan, \emph{Representations of the group
  {$\operatorname{GL}(n,K)$} where {$K$} is a local field}, Lie groups and
  their representations ({P}roc. {S}ummer {S}chool, {B}olyai {J}\'anos {M}ath.
  {S}oc., {B}udapest, 1971), Halsted, New York, 1975, pp.~95--118.

\bibitem{godement-jacquet}
R.~Godement and H.~Jacquet, \emph{Zeta functions of simple algebras}, Lecture
  Notes in Mathematics, Vol. 260, Springer-Verlag, Berlin-New York, 1972.

\bibitem{harris-taylor-langlands}
M.~Harris and R.~Taylor, \emph{The geometry and cohomology of some simple
  {S}himura varieties}, Annals of Mathematics Studies, vol. 151, Princeton
  University Press, Princeton, NJ, 2001, With an appendix by V. Berkovich.

\bibitem{jacquet-generic}
H.~Jacquet, \emph{Generic representations}, Non-commutative harmonic analysis
  ({A}ctes {C}olloq., {M}arseille-{L}uminy, 1976), Springer, Berlin, 1977,
  pp.~91--101. Lecture Notes in Math., Vol. 587.

\bibitem{jacquet-langlands}
H.~Jacquet and R.~Langlands, \emph{Automorphic forms on
  {$\operatorname{GL}(2)$}}, Lecture Notes in Mathematics, Vol. 114,
  Springer-Verlag, Berlin-New York, 1970.

\bibitem{jpss}
H.~Jacquet, I.~Piatetski-Shapiro, and J.~Shalika, \emph{Conducteur des
  repr\'esentations du groupe lin\'eaire}, Math. Ann. \textbf{256} (1981),
  no.~2, 199--214.

\bibitem{koch-zink}
H.~Koch and E.-W. Zink, \emph{Zur {K}orrespondenz von {D}arstellungen der
  {G}aloisgruppen und der zentralen {D}ivisionsalgebren \"uber lokalen
  {K}\"orpern (der zahme {F}all)}, Math. Nachr. \textbf{98} (1980), 83--119.

\bibitem{nps}
P.~Nelson, A.~Pitale, and A.~Saha, \emph{Bounds for {R}ankin--{S}elberg
  integrals and quantum unique ergodicity for powerful levels}, J. Amer. Math.
  Soc. \textbf{27} (2014), no.~1, 147--191.

\bibitem{prasad-raghuram}
D.~Prasad and A.~Raghuram, \emph{Representation theory of
  {$\operatorname{GL}(n)$} over non-{A}rchimedean local fields}, School on
  {A}utomorphic {F}orms on {$\operatorname{GL}(n)$}, ICTP Lect. Notes, vol.~21,
  Abdus Salam Int. Cent. Theoret. Phys., Trieste, 2008, pp.~159--205.

\bibitem{rogawski}
J.~Rogawski, \emph{Representations of {$\operatorname{GL}(n)$}\ and division
  algebras over a {$p$}-adic field}, Duke Math. J. \textbf{50} (1983), no.~1,
  161--196.

\bibitem{rohrlich}
D.~Rohrlich, \emph{Elliptic curves and the {W}eil--{D}eligne group}, Elliptic
  curves and related topics, CRM Proc. Lecture Notes, vol.~4, Amer. Math. Soc.,
  Providence, RI, 1994, pp.~125--157.

\bibitem{saha-sup-norm}
A.~Saha, \emph{On sup-norms of cusp forms of powerful level}, ArXiv e-prints
  (2014), Online.

\bibitem{saha-hybrid}
\bysame, \emph{Hybrid sup-norm bounds for maass newforms of powerful level},
  ArXiv e-prints (2015), Online.

\bibitem{saha-large-values}
\bysame, \emph{Large values of newforms on {$\operatorname{GL}(2)$} with highly
  ramified central character}, Int. Math. Res. Not. IMRN (2016), no.~13,
  4103--4131.

\bibitem{tignol-wadsworth}
J.-P. Tignol and A.~Wadsworth, \emph{Value functions on simple algebras, and
  associated graded rings}, Springer Monographs in Mathematics, Springer, Cham,
  2015.

\bibitem{tunnell}
J.~Tunnell, \emph{On the local {L}anglands conjecture for
  {$\operatorname{GL}(2)$}}, Invent. Math. \textbf{46} (1978), no.~2, 179--200.

\bibitem{wedhorn}
T.~Wedhorn, \emph{The local {L}anglands correspondence for
  {$\operatorname{GL}(n)$} over {$p$}-adic fields}, School on {A}utomorphic
  {F}orms on {$\operatorname{GL}(n)$}, ICTP Lect. Notes, vol.~21, Abdus Salam
  Int. Cent. Theoret. Phys., Trieste, 2008, pp.~237--320.

\bibitem{weil-basic}
A.~Weil, \emph{Basic number theory}, Classics in Mathematics, Springer-Verlag,
  Berlin, 1973, Second edition.

\bibitem{zelevinsky}
A.~Zelevinsky, \emph{Induced representations of reductive {$\mathfrak{p}$}-adic
  groups. {II}. {O}n irreducible representations of {$\operatorname{GL}(n)$}},
  Ann. Sci. \'Ecole Norm. Sup. (4) \textbf{13} (1980), no.~2, 165--210.

\end{thebibliography}

\end{document}